\documentclass[a4paper,12pt]{amsart}
\usepackage{amssymb,amsfonts,amsmath}
\usepackage{ifthen}
\usepackage{graphicx}
\usepackage{float}
\usepackage{cite}
\newtheorem{theorem}{Theorem}[section]
\newtheorem{lemma}[theorem]{Lemma}
\newtheorem{corollary}[theorem]{Corollary}
\newtheorem{proposition}[theorem]{Proposition}

\newtheorem{definition}[theorem]{Definition}
\newtheorem{remark}[theorem]{Remark}

\newtheorem{conjecture}[theorem]{Conjecture}

\numberwithin{equation}{section}

\newcounter{minutes}
\divide\time by 60
\newcounter{hours}
\multiply\time by 60
\addtocounter{minutes}{-\time}

\newcommand{\D}{{\mathbb D}}

\newcommand{\real}{{\operatorname{Re}\,}}

\newcommand{\Log}{{\operatorname{Log}\,}}
\newcommand{\Arg}{{\operatorname{Arg}\,}}
\newcommand{\ds}{\displaystyle}

\begin{document}

\bibliographystyle{amsplain}

\title[Coefficient estimates for some classes of functions associated with \(q\)-function theory]%
{Coefficient estimates for some classes of functions associated with \(q\)-function theory}

\def\thefootnote{}
\footnotetext{ \texttt{\tiny File:~\jobname .tex,
          printed: \number\day-\number\month-\number\year,
          \thehours.\ifnum\theminutes<10{0}\fi\theminutes}
} \makeatletter\def\thefootnote{\@arabic\c@footnote}\makeatother

\author{Sarita Agrawal}
\thanks{
Discipline of Mathematics,
Indian Institute of Technology Indore,
Simrol, Khandwa Road, Indore 453 552, India\\
{\em Email: saritamath44@gmail.com}\\
{\bf Acknowledgement}: I thank my PhD supervisor Dr. Swadesh Kumar Sahoo
for his valuable suggestions and careful reading of this manuscript.
}

\begin{abstract}
In this paper, for every $q\in(0,1)$, we obtain the Herglotz representation theorem and discuss the Bieberbach type problem for the class of 
$q$-convex functions of order $\alpha, 0\le\alpha<1$. In addition, we discuss the Fekete-szeg\"o problem and the Hankel 
determinant problem for the class of $q$-starlike functions, leading to couple of  
conjectures for the class of $q$-starlike functions of order $\alpha, 0\le\alpha<1$.


\end{abstract}

\maketitle
\pagestyle{myheadings}
\markboth{Sarita Agrawal}{Coefficient problems}

\section{Introduction}
Throughout the present investigation, we denote by $\mathbb{C}$, the set of complex numbers and by $\mathcal{H}(\D)$, the set of all analytic (or holomorphic) 
functions in $\D$. We use the symbol $\mathcal{A}$ for the class of functions 
$f \in \mathcal{H}(\D)$ with 
the standard normalization $f(0)=0=f'(0)-1$. i.e. the functions 
$f\in\mathcal{A}$ have the power series representation of the form 
\begin{equation}\label{e1}
f(z)=z+\sum_{n=2}^\infty a_nz^n.
\end{equation}
The set $\mathcal{S}$ denotes the class of {\em univalent} functions in $\mathcal{A}$.
We denote by $\mathcal{S}^*$ and $\mathcal{C}$, the class of starlike and convex
functions in $\mathcal{A}$ respectively. These are vastly available in the literature; 
see \cite{Dur83,Goo83}. The principal value of the logarithmic function $\log z$ for $z\neq 0$ is denoted by $\Log z:=\ln |z|+i \Arg(z)$, where $-\pi\le \Arg(z)<\pi$.

In geometric function theory, finding bound for the coefficient $a_n$ of functions of the form (\ref{e1}) is an important
problem, as it reveals the geometric properties of the corresponding function. For
example, the bound for the second coefficient $a_2$ of functions in the class $\mathcal{S}$, gives the growth and distortion properties as well as covering theorems. Bieberbach proposed a conjecture in the year $1916$ that 
{\em ``among all functions
in $\mathcal{S}$,
the Koebe function has the largest coefficient"}; for instance see \cite{Dur83,Goo83}.
This conjecture was a challenging open problem for mathematicians for
several decades. To prove this conjecture initial approach was made for some subclasses of univalent functions like $\mathcal{S}^*,\mathcal{C}$, etc. Many more new techniques were developed in order to settle the conjecture. One of the important techniques is the {\em Herglotz representation theorem} which tells about the integral representation of analytic functions with positive real part in $\D$. 
Finally, the complete proof of Bieberbach's conjecture was settled by de Branges in $1985$ \cite{deB85}. 

Another interesting coefficient estimation is the Hankel determinant. The $k^{th}$ order Hankel determinant ($k\ge 1$) of $f\in 
\mathcal{A}$ is defined by
$$H_k(n)=\left|\begin{array}{ccc}
a_n & a_{n+1} \cdots & a_{n+k-1}\\
a_{n+1} & \cdots & a_{n+k}\\
\vdots & \vdots &\vdots\\
a_{n+k-1} & \cdots & a_{n+2k-2}  
\end{array}
\right|.
$$
For our discussion, in this paper, we consider the Hankel determinant $H_2(1)$ (also called the Fekete-Szeg\"o functional) and $H_2(2)$.
Also in 1916, Bieberbach proved that if $f\in\mathcal{S}$, then $|a_2^2-a_3|\le 1$. In 1933, Fekete and Szeg\"o in \cite{FS33} proved that
$$|a_3-\mu a_2^2|\le \left \{ \begin{array}{ll}
4\mu-3 & \mbox{if } \mu \ge 1\\
1+2\exp [-2\mu/(1-\mu)] & \mbox{if } 0\le \mu \le 1\\
3-4\mu & \mbox{if } \mu \le 0
\end{array}\right. .
$$
The result is sharp in the sense that for each $\mu$ there is a function in
the class under consideration for which equality holds. 
The coefficient functional $a_3-\mu a_2^2$ has many applications in function theory. For example,
the functional $a_3-a_2^2$ is equal to $S_f(z)/6$, where $S_f(z)$ is the Schwarzian derivative 
of the locally univalent function $f$ defined by
$S_f(z)=(f''(z)/f'(z))'-(1/2)(f''(z)/f'(z))^2$.
Finding the maximum value of the functional $a_3-\mu a_2^2$ is called the
{\em Fekete-Szeg\"o problem}.
Koepf solved the Fekete-Szeg\"o problem for close-to-convex functions and obtains the largest real number $\mu$ for which $a_3-\mu a_2^2$ is maximized by the Koebe function $z/(1-z)^2$ is $\mu=1/3$ (see \cite{Koe87}). Later, in \cite{Koe87-II} (see also \cite{Lon93}), this result
was generalized for functions that are close-to-convex of order $\beta$, $\beta\ge 0$. In \cite{Pfl85}, Pfluger employed the variational method to give another treatment of the Fekete-Szeg\"o inequality which includes a description of the image domains under extremal functions. Later, Pfluger \cite{Pfl86} used Jenkin’s method to show that
for $f\in \mathcal{S}$,
$$|a_3-\mu a_2^2|\le 1+2|\exp(-2\mu/(1-\mu))|
$$
holds for complex $\mu$ such that $\real(1/(1-\mu))\ge 1$. The inequality is sharp if and only if $\mu$
is in a certain pear shaped subregion of the disk given by
$$\mu =1-(u+itv)/u^2+v^2, \quad -1\le t\le 1,
$$ 
where $u=1-\log(\cos \varphi)$ and $v=\tan \varphi- \varphi , 0<\varphi <\pi/2$.

In recent years, study of $q$-analogs of subclasses of univalent functions is well adopted among function
theorists. Bieberbach type problems for functions belonging to classes associated with $q$-function theory
are discussed in \cite{IMS90,AS14-2,SS15}. In the sequel, we discuss the Bieberbach type problem for $q$-analog of convex functions of order $\alpha,0\le \alpha<1$. 
Finding of Hankel determinant and Fekete-Szeg\"o problem for subclasses of univalent functions are vastly
available in literature, see, for instance \cite{KM69,Koe87, Koe87-II}. But these type of problems are not considered for 
classes involving $q$-theory. 
In this regard, we motivate to discuss the Hankel determinant and Fekete-Szeg\"o problems for the $q$-analog
of starlike functions.

\section{Preliminaries, and Main Theorems}\label{prelm}
For $0<q<1$, {\em the $q$-difference operator} (see \cite{AS14-2}), denoted as $D_qf$,  is defined by
the equation
$$(D_qf)(z)=\frac{f(z)-f(qz)}{z(1-q)},\quad z\neq 0, \quad (D_qf)(0)=f'(0).
$$
Now, recall the defintion of the class of {\em $q$-starlike functions of order $\alpha, 0\le\alpha<1$,} denoted by $\mathcal{S}_q^*(\alpha)$. 

\begin{definition}\cite[Definition~1.1]{AS14-2}
A function $f\in\mathcal{A}$ is said to be in the class $\mathcal{S}_q^*(\alpha)$, $0\le \alpha<1$, if 
$$\left|\frac{\ds\frac{z(D_qf)(z)}{f(z)}-\alpha}{1-\alpha}-\frac{1}{1-q}\right|\leq \frac{1}{1-q}, \quad z\in \mathbb{D}.
$$
\end{definition}

Note that
the choice $\alpha=0$ gives the definition of the class of $q$-starlike functions, denoted by $\mathcal{S}_q^*$, (see \cite[Definition~1.3]{IMS90}). Indeed, a function $f\in\mathcal{A}$ is said to belong to $\mathcal{S}_q^*$ if $|(z(D_qf)(z))/f(z)-1/(1-q)|\leq 1/(1-q),  z\in \mathbb{D}$.
By using the idea of the well-known Alexander's theorem \cite[Theorem~2.12]{Dur83}, Baricz and Swaminathan in \cite{BS14} defined a $q$-analog of convex functions, denoted by $\mathcal{C}_q$, in the following way.
\begin{definition}\cite[Definition~3.1]{BS14}
A function $f\in\mathcal{A}$ is said to belong to $\mathcal{C}_q$ if and only if $z(D_qf)(z)\in \mathcal{S}_q^*$.
\end{definition}
We call the functions of the class $\mathcal{C}_q$ as {\em $q$-convex functions}. The class $\mathcal{C}_q$ is non-empty as shown in \cite[Theorem~3.2]{BS14}. Note that as $q\to 1$, the classes 
$\mathcal{S}_q^*$ and $\mathcal{C}_q$ reduce to $\mathcal{S}^*$ and $\mathcal{C}$ respectively.

It is natural to define {\em the $q$-convex functions of order $\alpha$}, $0\le \alpha< 1$, denoted by 
$\mathcal{C}_q(\alpha)$, in the following way:
\begin{definition}\label{def}
A function $f\in\mathcal{A}$ is said to be in the class $\mathcal{C}_q(\alpha), 0\le \alpha<1$, if and only if $z(D_qf)(z)\in \mathcal{S}_q^*(\alpha)$.
\end{definition}
We can see that as $q\to 1$, the class $\mathcal{C}_q(\alpha)$ reduces to the class of convex functions of order $\alpha$, $\mathcal{C(\alpha)}$ (for definition of $\mathcal{C(\alpha)}$ see \cite{Goo83}).

Bieberbach type problem is estimated for the classes $\mathcal{S}_q^*$ and $\mathcal{S}_q^*(\alpha)$ in the articles \cite{IMS90} and \cite{AS14-2} respectively. 
But the Fekete-Szeg\"o problem and the Hankel determinant were not considered there.
In this article, we first discuss these two problems for the class $\mathcal{S}_q^*$ and posed two conjectures on the Fekete-Szeg\"o problem and Hankel determinant for the class $\mathcal{S}_q^*(\alpha)$. Since the Bieberbach type problem for the class $\mathcal{C}_q(\alpha), 0\le \alpha<1$, 
is not available in literature, here we obtain the Bieberbach type problem for the class $\mathcal{C}_q(\alpha)$ for $0\le \alpha<1$. In addition, we find the Herglotz representation formula for functions belonging to the class $\mathcal{C}_q(\alpha)$. One can also think of Hankel determinant, Fekete-Szeg\"o problems for $\mathcal{C}_q(\alpha)$ as well.   

The concept of $q$-integral is useful in this setting. 
Thomae was a pupil of Heine who introduced, the so-called, $q$-integral \cite{Tho69}
$$\int_0^1 f(t)\, d_q t = (1-q)\sum_{n=0}^{\infty} q^n f(q^n),
$$
provided the $q$-series converges.
In 1910, Jackson defined the general $q$-integral \cite{Jac10} (see also \cite{GR90,Tho69}) in the following manner:
$$\int_a^b f(t)\, d_q t:=\int_0^b f(t)\, d_q t - \int_0^a f(t)\, d_q t, 
$$
where
$$I_q(f(x)):=\int_0^x f(t)\, d_q t = x(1-q)\sum_{n=0}^{\infty} q^n f(xq^n),
$$
provided the $q$-series converges. Observe that
$$D_q I_q f(x)=f(x)~~\mbox{ and }~~
I_q D_q f(x)=f(x)-f(0),
$$
where the second equality holds if $f$ is continuous at $x=0$.
For more background on $q$-integrals, we refer to \cite{GR90}.

Now, we state our main results.
The Fekete-Szeg\"o problem for the class $\mathcal{S}_q^*$ is obtained as follows:
\begin{theorem}\label{T3}
Let $f\in\mathcal{S}_q^*$ be of the form (\ref{e1}) and $\mu$ be any complex number. Then
$$
|a_3-\mu a_2^2|\le \max\left\{\left|2(1-2\mu)\left(\frac{\ln q}{q-1}\right)^2+2\left(\frac{\ln q}{q^2-1}\right)\right|, 2\left(\frac{\ln q}{q^2-1}\right)\right\}.
$$
Equality occurs for the functions
\begin{equation}\label{e2}
F_1(z):=z\left\{\exp \left[\ds \sum_{n=1}^\infty \frac{2\ln q}{q^n-1}z^n\right]\right\}
\end{equation}
and
\begin{equation}\label{e3}
F_2(z):=z\left\{\exp \left[\ds \sum_{n=1}^\infty \frac{2\ln q}{q^{2n}-1}z^{2n}\right]\right\}.
\end{equation}
\end{theorem}

The next result is the estimation of second order Hankel determinant for the class $\mathcal{S}_q^*$.

\begin{theorem}\label{T4}
Let $f\in\mathcal{S}_q^*$ be of the form (\ref{e1}). Then
$$
|H_2(2)|=|a_2a_4-a_3^2|\le 4\left(\frac{\ln q}{q^2-1}\right)^2.
$$
Equality occurs for the function $F_2(z)$ defined in $(\ref{e3})$.
\end{theorem}

\begin{remark}
For $q\to 1$, Theorem~\ref{T3} gives the Fekete-Szeg\"o problem
for the class $\mathcal{S}^*$ \cite[Theorem~1]{KM69}.
\end{remark}

\begin{remark}
For $q\to 1$, Theorem~\ref{T4} gives the Hankel determinant
for the class $\mathcal{S}^*$ \cite[Theorem~3.1]{Jan07}.
\end{remark}

Now we present the Herglotz representation of functions belonging to the class $\mathcal{C}_q(\alpha)$:

\begin{theorem}\label{thm2}
Let $f\in \mathcal{A}$. Then $f\in\mathcal{C}_q(\alpha)$, $0\le \alpha<1$, if and only if there exists 
a probability measure $\mu$ supported on the unit circle such that
$$\frac{z(D_qf)'(z)}{(D_qf)(z)}=\int_{|\sigma|=1}\sigma z F_{q, \alpha}^{'}(\sigma z)\rm{d}\mu(\sigma)
$$
where
\begin{equation}\label{MainThm1:eq1}
F_{q,\alpha}(z)=\ds \sum_{n=1}^\infty \frac{(-2)\left(\ln \frac{q}{1-\alpha(1-q)}\right)}{1-q^n}z^n, \quad z\in \D.
\end{equation}
\end{theorem}

\begin{remark}
It is clear that when $q\to1$, 
$$F_{q, \alpha}^{'}(z)\to 2(1-\alpha)/(1-z) \mbox{ and } z(D_qf)'(z)/(D_qf)(z)\to zf''(z)/f'(z).
$$ 
Hence, when $q$ approaches to $1$, Theorem~\ref{thm2} leads to the Herglotz Representation of 
convex functions of order $\alpha$ (see for instance \cite[pp. 172, Problem~3]{Goo83}).
\end{remark}

The Bieberbach type problem for the class $\mathcal{C}_q(\alpha)$, $0\le \alpha<1$, is stated below: 

\begin{theorem}\label{sec2-thm7}
Let 

\begin{equation}\label{MainThm2:eq}
E_q(z):=I_q\{\exp [F_{q,\alpha}(z)]\}=z+\ds \sum_{n=2}^\infty\left(\frac{1-q}{1-q^n}\right) c_n z^n
\end{equation}
where $c_n$ is the $n$-th coefficient of the function $z\exp [F_{q,\alpha}(z)]$. 
Then $E_q\in \mathcal{C}_q(\alpha)$, $0\le \alpha<1$. Moreover, if $f(z)=z+\sum_{n=2}^\infty a_n z^n\in 
\mathcal{C}_q(\alpha)$, then $|a_n|\le((1-q)/(1-q^n)) c_n$ with equality holding for all $n$ if and only if 
$f$ is a rotation of $E_q$.
\end{theorem}

\begin{remark}
It would be interesting to get an explicit form of the extremal function  independent of the $q$-integral
in Theorem~\ref{sec2-thm7}. 
\end{remark}

\begin{remark}
It is clear that when $q\to 1$, 
$$F_{q,\alpha}(z)\to -2(1-\alpha)\log(1-z),
$$ 
and hence $z\exp [F_{q,\alpha}(z)]\to z/(1-z)^{2(1-\alpha)}$. Therefore, as $q\to 1$, the coefficient $c_n\to {\prod_{k=2}^n (k-2\alpha)}/(n-1)!$, which gives $|a_n|$ is bounded by ${\prod_{k=2}^n (k-2\alpha)}/n!$ for $f\in \mathcal{C}(\alpha)$. i.e. when $q\to 1$, Theorem~\ref{sec2-thm7} leads to the Bieberbach type problem for the class $\mathcal{C}(\alpha)$
(see for instance \cite[Theorem~2, pp. 140]{Goo83}).
\end{remark}

\section{Properties of the class $\mathcal{C}_q(\alpha), 0\le \alpha<1$}\label{sec2}
This section is devoted to study of some basic properties of the class $\mathcal{C}_q(\alpha)$. The following proposition says that a function
$f\in\mathcal{C}_q(\alpha)$ can be written in terms of a function $g$ in $\mathcal{S}^*_q(\alpha)$. The proof is obvious and it follows
from the definition of $\mathcal{C}_q(\alpha)$.
 
\begin{proposition}\label{sec1-prop1}
Let $f \in \mathcal{C}_q(\alpha)$, $0\le \alpha<1$. Then there exists a unique function $g \in \mathcal{S}^*_q(\alpha)$, $0\le \alpha<1$, such that 
\begin{equation}\label{prop1-eqn}
g(z)=z(D_qf)(z)
\end{equation}
holds. Similarly, for a given function $g\in \mathcal{S}_q^*(\alpha)$ there exists 
a unique function $f\in \mathcal{C}_q(\alpha)$ satisfying $(\ref{prop1-eqn})$. 
\end{proposition}

Next result is a characterization for a function to be in the class $\mathcal{C}_q(\alpha)$.
\begin{theorem}\label{sec2-thm1}
Let $f\in \mathcal{A}$. Then $f\in\mathcal{C}_q(\alpha)$, $0\le \alpha<1$, if and only if
$$\left|q\frac{(D_qf)(qz)}{(D_qf)(z)}-\alpha q\right|\leq {1-\alpha}, \quad z\in \mathbb{D}.
$$
\end{theorem}

\begin{proof}
By Definition~\ref{def}, we have $f\in\mathcal{C}_q(\alpha)$ if and only if $z(D_qf)(z)\in \mathcal{S}_q^*(\alpha)$. Then the result follows immediately from \cite[Theorem~2.2]{AS14-2}.
\end{proof}

\begin{corollary}
The class $\mathcal{C}_q(\alpha)$ satisfies the inclusion relation 
$$\bigcap_{q<p<1}\mathcal{C}_p(\alpha)\subseteq \mathcal{C}_q(\alpha)
~~\mbox{ and }~~
\bigcap_{0<q<1}\mathcal{C}_q(\alpha) = \mathcal{C}(\alpha).
$$ 
\end{corollary}

\begin{proof} 
If $f\in\mathcal{C}_p(\alpha)$ for all $p\in (q,1)$, then as $p\to q$ we get $f\in\mathcal{C}_q(\alpha)$. Hence the inclusion 
$$\bigcap_{q<p<1}\mathcal{C}_p(\alpha)\subseteq \mathcal{C}_q(\alpha)
$$
holds.
Similarly, if $f\in\mathcal{C}_q(\alpha)$ for all $q\in (0,1)$, then as $q\to 1$ we get $f\in\mathcal{C}(\alpha)$. That is,
$$
\bigcap_{0<q<1}\mathcal{C}_q(\alpha) \subseteq \mathcal{C}(\alpha)
$$ 
holds. It remains to show that 
$$\mathcal{C}(\alpha)\subseteq \bigcap_{0<q<1}\mathcal{C}_q(\alpha).
$$
For this, we let $f\in \mathcal{C}(\alpha)$. Then we show that $f\in \mathcal{C}_q(\alpha)$ for all $q\in (0,1)$.
Since $f\in \mathcal{C}(\alpha)$, $zf'\in \mathcal{S}^*(\alpha)$. By \cite[Corollary~2.3]{AS14-2}, $\mathcal{S}^*(\alpha)=\cap_{0<q<1}\mathcal{S}^*_q(\alpha)$, it follows that $zf'\in \mathcal{S}^*_q(\alpha)$ for all $q\in(0,1)$. 

Thus, by Proposition~\ref{sec1-prop1}, there exists a unique $h\in \mathcal{C}_q(\alpha)$ satisfying the 
identity (\ref{prop1-eqn}) with $h(z)=f(z)$. The proof now follows immediately.
\end{proof}

We now define two sets and proceed to prepare some basic results which are being used 
to prove our main results as well. Define
$$B_q=\{g:g\in \mathcal{H}(\D),~g(0)=q \mbox{ and } g:\D \to \D\}
~~\mbox{ and }~~
B_q^0=\{g:g\in B_q \mbox{ and } 0\notin g(\D) \}.
$$

\begin{lemma}{\label{lm2}} \cite[Lemma~2.4]{AS14-2}
If $h\in B_q$ then the infinite product 
$\prod_{n=0}^\infty \{((1-\alpha)h(zq^n)+\alpha q)/q\}$ converges uniformly on compact subsets of $\D$.
\end{lemma}

\begin{lemma}{\label{lm3}}
If $h\in B_q^0$ then the infinite product $\prod_{n=0}^\infty \{((1-\alpha)h(zq^n)+\alpha q)/q\}$ converges 
uniformly on compact subsets of $\D$ to a nonzero function in $\mathcal{H}(\D)$ with no zeros. Furthermore, the function $f$ satisfying the relation
\begin{equation}\label{eq3}
z(D_qf)(z)=\frac{z}{\prod_{n=0}^\infty \{((1-\alpha)h(zq^n)+\alpha q)/q\}}
\end{equation}
belongs to $\mathcal{C}_q(\alpha)$ and $h(z)=\ds \left(q\frac{(D_qf)(qz)}{(D_qf)(z)}-\alpha q\right)/(1-\alpha)$.
\end{lemma}
\begin{proof}
The convergence of the infinite product is due to Lemma \ref{lm2}. 
Since $h\in B_q^0$, we have $h(z)\neq 0$ in 
$\D$ and the infinite product does not vanish in $\D$. Thus, the function $z(D_qf)(z)\in \mathcal{A}$ and
we find the relation
$$ q\frac{(D_qf)(qz)}{(D_qf)(z)}=q\lim_{k\to\infty}\prod_{n=0}^k\frac{(1-\alpha)h(zq^n)+\alpha q}{(1-\alpha)h(zq^{n+1})+\alpha q}
=(1-\alpha)h(z)+\alpha q.
$$
Since $h\in B_q^0$, we get $f\in \mathcal{C}_q(\alpha)$ and the proof of our lemma is complete.
\end{proof}

Let $\mathcal{P}$ be the family of all functions $p\in \mathcal{H}(\D)$ 
for which $\real \{p(z)\}\ge 0$ and
\begin{equation}\label{e6}
p(z)=1+p_1z+p_2z^2+\ldots
\end{equation}
for $z\in\D$.

\begin{lemma}\label{lm}\cite[Lemma~2.4]{IMS90}
A function $g\in B_q^0$ if and only if it has the representation
\begin{equation}\label{eq4}
g(z)=\exp\{(\ln q) p(z)\}, 
\end{equation}
where $p(z)$ belongs to the class $\mathcal{P}$.
\end{lemma}

\begin{theorem}\label{thm1}
The mapping $\rho:\mathcal{C}_q(\alpha) \to B_q^0$ defined by
$$\rho(f)(z)=\left(q\frac{(D_qf)(qz)}{(D_qf)(z)}-\alpha q\right)/(1-\alpha)
$$
is a bijection.
\end{theorem}

\begin{proof}
For $ h \in B_q^0 $, define a mapping $\sigma:\,B_q^0 \to \mathcal{A}$ by 
$$
z(D_q\sigma(h))(z)=\frac{z}{\prod_{n=0}^\infty \{((1-\alpha)h(zq^n)+\alpha q)/q\}}
$$
It is clear from Lemma~\ref{lm3} that $\sigma(h) \in \mathcal{C}_q$ and $(\rho\circ \sigma)(h)=h$. Considering the composition mapping $\sigma\circ \rho$ we compute that

\begin{eqnarray*}
z(D_q(\sigma\circ \rho)(f))(z)
&=&\frac{z}{\prod_{n=0}^\infty \{((1-\alpha)\rho(f)(zq^n)+\alpha q)/q\}}\\
&=&\frac{z}{\prod_{n=0}^\infty \{q(D_qf)(zq^{n+1})/q(D_qf)(zq^n)\}}
=z(D_qf)(z)
\end{eqnarray*}
or,
$$
(\sigma\circ \rho)(f)=f.
$$
Hence $\sigma\circ \rho$ and $\rho\circ \sigma$ are identity mappings and $\sigma$ 
is the inverse of $\rho$, i.e. the map $\rho(f)$ is invertible. Hence $\rho(f)$ is a bijection. This completes the proof of our theorem.
\end{proof}
\section{Proof of the main theorems}

In this section we prove our main theorems stated in Section~\ref{prelm}.
The following lemmas are useful for the proof of the Fekete-Szeg\"o problem and finding the Hankel determinant.
\begin{lemma}\label{l1}\cite[Theorem~1.13]{IMS90}
The mapping $\rho:\mathcal{S}_q^* \to B_q^0$ defined by
$$\rho(f)(z)=\ds \frac{f(qz)}{f(z)}
$$
is a bijection.
\end{lemma}

\begin{lemma}\label{l3}\cite[Theorem~1.15]{IMS90}
Let $f\in \mathcal{A}$. Then $f\in\mathcal{S}_q^*$ if and only if there exists 
a probability measure $\mu$ supported on the unit circle such that
$$\frac{zf'(z)}{f(z)}=1+\int_{|\sigma|=1}\sigma z F_q^{'}(\sigma z)\rm{d}\mu(\sigma)
$$
where
\begin{equation}
F_q(z)=\ds \sum_{n=1}^\infty \frac{2\ln q}{q^n-1}z^n, \quad z\in \D .
\end{equation}
\end{lemma}

\begin{lemma}\label{l4}\cite[pp.~254-256]{LZ83}
Let the function $p\in\mathcal{P}$ and be given by the power series (\ref{e6}). Then
$$
2p_2=p_1^2+x(4-p_1^2),
$$
$$4p_3=p_1^3+2(4-p_1^2)p_1x-p_1(4-p_1^2)x^2+2(4-p_1^2)(1-|x|^2)z,
$$
for some $x$ and $z$ satisfying $|x|\le 1$, $|z|\le 1$, and $p_1\in [0,2]$.
\end{lemma}

\begin{lemma}\label{l5}\cite[Lemma~1]{MM92}
Let the function $p\in\mathcal{P}$ and be given by the power series (\ref{e6}). Then for any real number $\lambda$,
$$
|p_2-\lambda p_1^2|\le 2 \max\{1, |2\lambda-1|\}
$$
and the result is sharp.
\end{lemma}

\begin{proof}[\bf Proof of Theorem~\ref{T3}]
Let $f\in \mathcal{S}^*_q$.
Then by Lemma~\ref{l1}, there exist a function $g\in B_q ^0$ such that 
$g(z)=f(qz)/f(z)$. 
Since $g\in B_q ^0$, by Lemma~\ref{lm}, $g(z)$ has the representation (\ref{eq4}). That is,
$$ \frac{f(qz)}{f(z)}=\exp\{(\ln q)p(z)\}.
$$
Define the function $\phi(z)=\Log\{f(z)/z\}$ and set
$$\phi(z)=\Log\frac{f(z)}{z}=\sum_{n=1}^\infty \phi_n z^n.
$$
On solving, we get 
$$
\ln q+\phi(qz)=\phi(z)+(\ln q)p(z).
$$
This implies
\begin{equation}\label{e5}
\phi_n=p_n\left(\frac{\ln q}{q^n-1}\right).
\end{equation}
So, $f(z)$ can be written as
\begin{equation}\label{e5.5}
f(z)=z\exp\left[\sum_{n=1}^\infty \phi_n z^n\right],
\end{equation}
where $\phi_n$ is defined in (\ref{e5}) and $f(z)$ has the form (\ref{e1}).
Equating the coefficients of both sides in (\ref{e5.5}) and using the value of $\phi_n$ given in $(\ref{e5})$, we obtain
\begin{equation}\label{e5.6}
a_2=\phi_1=p_1\left(\frac{\ln q}{q-1}\right),\quad
a_3=\phi_2+\frac{\phi_1^2}{2}=p_2\left(\frac{\ln q}{q^2-1}\right)+\frac{p_1^2}{2}\left(\frac{\ln q}{q-1}\right)^2.
\end{equation}
Thus,
\begin{eqnarray*}
|a_3-\mu a_2^2|&=&\left|p_2\left(\frac{\ln q}{q^2-1}\right)+\frac{p_1^2}{2}\left(\frac{\ln q}{q-1}\right)^2-\mu p_1^2 \left(\frac{\ln q}{q-1}\right)^2\right|\\
&=&\left(\frac{\ln q}{q^2-1}\right)\left|p_2-(2\mu -1)\ds\frac{\ds\left(\frac{\ln q}{q-1}\right)^2}{\ds\left(\frac{2\ln q}{q^2-1}\right)}p_1^2\right|\\
&\le& \max\left\{\left|2(1-2\mu)\left(\frac{\ln q}{q-1}\right)^2+2\left(\frac{\ln q}{q^2-1}\right)\right|, 2\left(\frac{\ln q}{q^2-1}\right)\right\},
\end{eqnarray*}
where the last inequality follows from Lemma~\ref{l5}. It now remains to prove the sharpness part.

This can easily be shown by the definition of $\mathcal{S}^*_q$ that the functions $F_1$ and $F_2$ defined in the statement of Theorem~\ref{T3} belong to the class $\mathcal{S}^*_q$. One can also see that $F_1 \in \mathcal{S}^*_q$ as a special case to Lemma~\ref{l3}, when the measure has a unit mass. The functions $F_1$ and $F_2$ show the sharpness of the result. This completes the proof of the theorem.
\end{proof}

We now pose the following conjecture on Fekete-Szeg\"o problem for $\mathcal{S}_q^*(\alpha)$. 
\begin{conjecture}\label{sqsa-f}
Let $f\in\mathcal{S}_q^*(\alpha)$, $0\le\alpha<1$, be of the form (\ref{e1}) and $\mu$ be any complex number. Then
$$
|a_3-\mu a_2^2|\le \max\left\{\left|2(1-2\mu)\left(\frac{\ln \frac{q}{1-\alpha(1-q)}}{q-1}\right)^2+2\left(\frac{\ln \frac{q}{1-\alpha(1-q)}}{q^2-1}\right)\right|,2\left(\frac{\ln \frac{q}{1-\alpha(1-q)}}{q^2-1}\right)\right\}.
$$
Equality occurs for the functions
\begin{equation}\label{e7}
F_1(z):=z\left\{\exp \left[\ds \sum_{n=1}^\infty \frac{2\ln \frac{q}{1-\alpha(1-q)}}{q^n-1}z^n\right]\right\}
\end{equation}
and
\begin{equation}\label{e8}
F_2(z):=z\left\{\exp \left[\ds \sum_{n=1}^\infty \frac{2\ln \frac{q}{1-\alpha(1-q)}}{q^{2n}-1}z^{2n}\right]\right\}.
\end{equation}
\end{conjecture}

\begin{proof}[\bf Proof of Theorem~\ref{T4}]
Given that $f\in \mathcal{S}_q^*$ having the form \eqref{e1}. In (\ref{e5.6}), we already obtained the values of $a_2$ and $a_3$. In the similar way one can find the value of
$a_4$. Indeed,

$$a_4=\phi_3+\phi_1\phi_2+\frac{\phi_1^3}{6}=p_3\left(\frac{\ln q}{q^3-1}\right)+p_1p_2\left(\frac{\ln q}{q-1}\right)\left(\frac{\ln q}{q^2-1}\right)+\frac{p_1^3}{6}\left(\frac{\ln q}{q-1}\right)^3.
$$
Hence,
$$|a_2a_4-a_3^2|=\left|-\frac{p_1^4}{12}\left(\frac{\ln q}{q-1}\right)^4+p_1p_3\left(\frac{\ln q}{q-1}\right)\left(\frac{\ln q}{q^3-1}\right)-p_2^2\left(\frac{\ln q}{q^2-1}\right)^2\right|.
$$

Suppose now that $p_1=c$ and $0\le c\le 2$. Using Lemma~\ref{l4}, we obtain
\begin{eqnarray*}
|a_2a_4-a_3^2|&=&\left|-\frac{c^4}{12}\left[\left(\frac{\ln q}{q-1}\right)^4-3\left(\frac{\ln q}{q-1}\right)\left(\frac{\ln q}{q^3-1}\right)+3\left(\frac{\ln q}{q^2-1}\right)^2\right]\right.\\
&& \left.+\frac{c^2}{2}(4-c^2)x\left[\left(\frac{\ln q}{q-1}\right)\left(\frac{\ln q}{q^3-1}\right)-\left(\frac{\ln q}{q^2-1}\right)^2\right]\right.\\
&&\left.+\frac{(4-c^2)(1-|x|^2)cz}{2}\left(\frac{\ln q}{q-1}\right)\left(\frac{\ln q}{q^3-1}\right)\right.\\
&&\left.-\left[\frac{c^2}{4}(4-c^2)\left(\frac{\ln q}{q-1}\right)\left(\frac{\ln q}{q^3-1}\right)+\frac{(4-c^2)^2}{4}\left(\frac{\ln q}{q^2-1}\right)^2\right]x^2\right|\\
&\le & \frac{c^4}{12}\left|\left(\frac{\ln q}{q-1}\right)^4-3\left(\frac{\ln q}{q-1}\right)\left(\frac{\ln q}{q^3-1}\right)+3\left(\frac{\ln q}{q^2-1}\right)^2\right|
 +\frac{(4-c^2)c}{2}\left(\frac{\ln q}{q-1}\right)\\
&& \left(\frac{\ln q}{q^3-1}\right)+\frac{c^2}{2}(4-c^2)\left[\left(\frac{\ln q}{q-1}\right)\left(\frac{\ln q}{q^3-1}\right)-\left(\frac{\ln q}{q^2-1}\right)^2\right]\rho\\
&&+\left(\frac{4-c^2}{4}\right)\left[(4-c^2)\left(\frac{\ln q}{q^2-1}\right)^2+c(c-2)\left(\frac{\ln q}{q-1}\right)\left(\frac{\ln q}{q^3-1}\right)\right]\rho^2\\
&=& F(\rho)
\end{eqnarray*}
with $\rho=|x|\le 1$. Furthermore,
$$
F'(\rho)\ge 0.
$$
This implies that $F$ is an increasing function of $\rho$ and thus the upper bound for $|a_2a_4-a_3^2|$ corresponds to $\rho=1$.
Hence,
$$|a_2a_4-a_3^2|\le F(1)=G(c)\, \mbox{ (say)}.
$$
We can see that 
$$\left(\frac{\ln q}{q-1}\right)^4-3\left(\frac{\ln q}{q-1}\right)\left(\frac{\ln q}{q^3-1}\right)+3\left(\frac{\ln q}{q^2-1}\right)^2> 0, \mbox{ for } 0<q<1.
$$
Now, a simple calculation gives that
\begin{eqnarray*}
G(c)=\frac{c^4}{12}\left[\left(\frac{\ln q}{q-1}\right)^4-12\left(\frac{\ln q}{q-1}\right)\left(\frac{\ln q}{q^3-1}\right)+12\left(\frac{\ln q}{q^2-1}\right)^2\right]\\
&& \hspace*{-7cm} +c^2\left[3\left(\frac{\ln q}{q-1}\right)\left(\frac{\ln q}{q^3-1}\right)-4\left(\frac{\ln q}{q^2-1}\right)^2\right]+4\left(\frac{\ln q}{q^2-1}\right)^2.
\end{eqnarray*}
The expression $G'(c)=0$ gives either $c=0$ or 
$$c^2=\frac{6\ds \left[4\left(\frac{\ln q}{q^2-1}\right)^2-3\left(\frac{\ln q}{q-1}\right)\left(\frac{\ln q}{q^3-1}\right)\right]}{\ds \left(\frac{\ln q}{q-1}\right)^4-12\left(\frac{\ln q}{q-1}\right)\left(\frac{\ln q}{q^3-1}\right)+12\left(\frac{\ln q}{q^2-1}\right)^2}.
$$
We can verify that $G''(c)$ is negative for $c=0$ and positive for other values of $c$. Hence the maximum of $G(c)$ occurs at $c=0$. Thus, we obtain
$$
|a_2a_4-a_3^2|\le 4\left(\frac{\ln q}{q^2-1}\right)^2.
$$
The function $F_2$ defined in the statement of the theorem shows the sharpness of the result. 
This completes the proof of the theorem.
\end{proof}

We pose one more conjecture which is about the Hankel Determinant for the class $\mathcal{S}_q^*(\alpha)$. 
\begin{conjecture}\label{sqsa-h}
Let $f\in\mathcal{S}_q^*(\alpha)$, $0\le \alpha<1$, be of the form (\ref{e1}). Then
$$|a_2a_4-a_3^2|\le  4\left(\frac{\ln \frac{q}{1-\alpha(1-q)}}{q^2-1}\right)^2.
$$
Equality occurs for the function $F_2$ defined in $(\ref{e8})$.
\end{conjecture}

\begin{remark}
Here we remark that the proofs of Conjectures~\ref{sqsa-f} and ~\ref{sqsa-h} will follow in the 
similar manner as the proofs of Theorem~\ref{T3} and Theorem~\ref{T4}, respectively. However, the conjectures are all about to find the extremal functions which we believe to be \eqref{e7} and \eqref{e8}.
\end{remark}

\begin{proof}[\bf Proof of Theorem~\ref{thm2}]
Let $f\in \mathcal{C}_q(\alpha)$, $0\le \alpha<1$. By definition of $\mathcal{C}_q(\alpha)$, $z(D_qf)(z)\in \mathcal{S}_q^*(\alpha)$. Then by \cite[Theorem~1.1]{AS14-2}, we have
$$z\frac{(z(D_qf)(z))'(z)}{z(D_qf)(z)}=1+\int_{|\sigma|=1}\sigma z F_{q,\alpha}^{'}(\sigma z)\rm{d}\mu(\sigma)
$$
or,
$$1+\frac{z(D_qf)'(z)}{(D_qf)(z)}=1+\int_{|\sigma|=1}\sigma z F_{q,\alpha}^{'}(\sigma z)\rm{d}\mu(\sigma),
$$
where $F_{q,\alpha}$ is defined in (\ref{MainThm1:eq1}).
Hence the proof is complete.
\end{proof}

\begin{proof}[\bf Proof of Theorem~\ref{sec2-thm7}]
Let $f(z)=z+\sum_{n=2}^\infty a_n z^n\in \mathcal{C}_q(\alpha)$. By definition of $\mathcal{C}_q(\alpha)$, $z(D_qf)(z)=z+\sum_{n=2}^\infty (1-q^n)/(1-q)a_n z^n\in \mathcal{S}_q^*(\alpha)$. Then by \cite[Theorem~1.3]{AS14-2}, we have
$$\left|\frac{1-q^n}{1-q}a_n\right|\leq c_n.
$$
Next, we show that equality holds for the function $E_{q}\in \mathcal{C}_q(\alpha)$. As a special case to Theorem~\ref{thm2}, when the measure has a unit mass, it is clear that $E_{q}\in \mathcal{C}_q(\alpha)$. Let $E_q(z)=z+\sum_{n=2}^\infty b_n z^n$. From this representation of $E_q$ and the definition of $D_qf$, we get
\begin{equation}\label{e9}
z(D_q E_q)(z)=z+\sum_{n=2}^\infty b_n(1-q^n)/(1-q) z^n.
\end{equation} 
Since $E_q(z)=I_q\{\exp [F_{q,\alpha}(z)]\}$, $z(D_q E_q)(z)=z\{\exp [F_{q,\alpha}(z)]\}$ and since $c_n$ is the $n$-th coefficient of the function $z\exp [F_{q,\alpha}(z)]$, we have
\begin{equation}\label{e10}
z(D_q E_q)(z)=z+\sum_{n=2}^\infty c_n z^n.
\end{equation}
By comparing (\ref{e9}) and (\ref{e10}) we get, $b_n=c_n(1-q)/(1-q^n)$. i.e.
$$E_q(z)=z+\ds \sum_{n=2}^\infty\left(\frac{1-q}{1-q^n}\right) c_n z^n.
$$
This completes the proof of our theorem.
\end{proof}

\end{document}